\documentclass[11pt]{article}

\usepackage{fullpage}
\usepackage{amsmath,amsthm,amsfonts}
\usepackage{amssymb,latexsym,graphicx}
\usepackage{fourier}
\usepackage{mathtools}
\usepackage{hyperref}
\usepackage{graphicx}
\usepackage{url}

\makeatletter
\def\resetMathstrut@{%
  \setbox\z@\hbox{%
    \mathchardef\@tempa\mathcode`\(\relax
    \def\@tempb##1"##2##3{\the\textfont"##3\char"}%
    \expandafter\@tempb\meaning\@tempa \relax
  }%
  \ht\Mathstrutbox@1.2\ht\z@ \dp\Mathstrutbox@1.2\dp\z@
}
\makeatother
\newtheorem{theorem}{Theorem}[section]

\newtheorem{remark}{Remark}

\newtheorem{lemma}[theorem]{Lemma}
\newtheorem{lem}[theorem]{Lemma}

\newtheorem{corollary}[theorem]{Corollary}

\newcommand{\N}{\ensuremath{\mathbb{N}}}

\newcommand{\R}{\ensuremath{\mathbb{R}}}
\newcommand{\Z}{\ensuremath{\mathbb{Z}}}
\newcommand{\eqdef}{\stackrel{\mathrm{def}}{=}}
\renewcommand{\subset}{\subseteq}

\newcommand{\M}{\mathsf{M}}
\renewcommand{\setminus}{\smallsetminus}

\renewcommand{\epsilon}{\varepsilon}
\renewcommand{\le}{\leqslant}
\renewcommand{\ge}{\geqslant}
\renewcommand{\leq}{\leqslant}
\renewcommand{\geq}{\geqslant}
\newcommand{\Nsimplex}{\triangle^{\!N-1}}

\newcommand{\E}{\mathbb{E}}
\newcommand{\e}{\varepsilon}

\begin{document}
\title{Concentration of Markov chains with bounded moments}

\author{
Assaf Naor\thanks{Mathematics Department,  Princeton University. Supported by the Packard Foundation and the Simons Foundation.  The research that is presented here was conducted under the auspices of the Simons Algorithms and Geometry (A\&G) Think Tank.
}
\and
Shravas Rao\thanks{Courant Institute of Mathematical Sciences, New York
 University. This material is based upon work supported by the National Science Foundation Graduate Research Fellowship Program under Grant No. DGE-1342536.
}
\and
Oded Regev\thanks{Courant Institute of Mathematical Sciences, New York
 University. Supported by the Simons Collaboration on Algorithms and Geometry and by the National Science Foundation  under Grant No.~CCF-1814524. Any opinions, findings, and conclusions or recommendations expressed in this material are those of the authors and do not necessarily reflect the views of the NSF.}
}
\date{}
\maketitle

%%%% DON'T REMOVE %%%%% IF YOU WANT TO TURN OFF MACROS, DO IT IN THE MARGINNOTES.TEX FILE
%\noteswarning
%%%% DON'T REMOVE %%%%%

%\vspace{-0.15in}

\begin{abstract}
Let $\{W_t\}_{t=1}^{\infty}$ be a finite state stationary Markov chain, and suppose that $f$ is a real-valued function on the state space.
If $f$ is bounded, then Gillman's {\em expander Chernoff bound} (1993) provides  concentration estimates for the random variable $f(W_1)+\cdots+f(W_n)$ that depend on the spectral gap of the Markov chain and the assumed bound on $f$. Here we obtain analogous  inequalities assuming only that the $q$'th moment of $f$ is bounded for some $q \geq 2$. Our proof  relies on reasoning that differs substantially from the proofs of Gillman's theorem that are available in the literature, and it generalizes to yield dimension-independent bounds for mappings $f$ that take values in an $L_p(\mu)$ for some $p\ge 2$, thus answering (even in the Hilbertian special case $p=2$) a question of Kargin (2007).
\end{abstract}

\section{Introduction}

For  $N\in \N$, write $[N]\eqdef \{1,\ldots,N\}$ and let  $\Nsimplex\eqdef\big\{\pi=(\pi_1,\ldots,\pi_N)\in [0,1]^N:\ \sum_{i=1}^N\pi_i=1\big\}$ be the simplex of probability measures on $[N]$. Given $\pi\in \Nsimplex$, denote  by $E_\pi\in \M_N(\R)$ the  $N$-by-$N$ matrix all of whose rows equal $\pi$, i.e., $E_\pi u=(\sum_{j=1}^N \pi_ju_j,\ldots,\sum_{j=1}^N \pi_ju_j)\in \R^N$ for every $u=(u_1,\ldots,u_N)\in \R^N$.
%When we say that a matrix $A=(a_{ij})\in \M_N(\R)$ is stochastic we always mean  row-stochastic, i.e., $\{(a_{i1},\ldots,a_{iN})\}_{i=1}^N\subset %\Nsimplex$.

Given $\pi\in \Nsimplex$, a stochastic matrix $A=(a_{ij})\in \M_N(\R)$ is $\pi$-stationary if $\pi A=\pi$, i.e.,  $\pi_i=\sum_{j=1}^N\pi_ja_{ji}$ for all $i\in [N]$. We then define $\lambda_\pi(A)$ to be the norm of $A-E_\pi$  as an operator from $L_2(\pi)$ to $L_2(\pi)$, i.e.,
$$
\lambda_\pi(A)\eqdef \|A-E_\pi\|_{L_2(\pi)\to L_2(\pi)}=\sup\Bigg\{ \bigg(\sum_{i=1}^N \pi_i\Big(\sum_{j=1}^Na_{ij}u_j-\sum_{k=1}^N \pi_ku_k\Big)^2\bigg)^{\frac12}:\ u\in \R^N\ \mathrm{and}\ \sum_{k=1}^N \pi_ku_k^2=1\Bigg\}.
$$
Note that if $A$ is diagonalizable over the Hilbert space $L_2(\pi)$, then we have $\lambda_\pi(A)=\max\{\lambda_2(A),|\lambda_N(A)|\}$, where $1=\lambda_1(A)\ge \cdots\ge \lambda_N(A)\ge -1$ are the eigenvalues of $A$. This would occur if $A$ were $\pi$-reversible, i.e., $\pi_i a_{ij}=\pi_j a_{ji}$ for all $i,j\in [N]$, in which case $A$ would be a self-adjoint operator on $L_2(\pi)$; the reversible setting is the main case of interest in the ensuing discussion, but reversibility is not needed for our proofs.

Let $\mathbf{W}=\{W_t\}_{t=1}^\infty$ be a Markov chain with state space $[N]$ and transition matrix  $A\in \M_N(\R)$. One says that $\mathbf{W}$ is stationary if $A$ is $\pi_\mathbf{W}$-stationary for $\pi_\mathbf{W}=(\Pr[W_1=1],\ldots,\Pr[W_1=N])\in \Nsimplex$. Write $\lambda_\mathbf{W}=\lambda_{\pi_\mathbf{W}}(A)$.
\begin{theorem}\label{momentbound}
Suppose that $\mathbf{W}=\{W_t\}_{t=1}^{\infty}$ is a  stationary Markov chain whose state space is $[N]$ and  with $\lambda_\mathbf{W}<1$. Then, every $f:[N]\to \R$ satisfies the following inequality for every $n\in \N$ and every $q\ge 2$.
\begin{equation}\label{eq:q in rhs1}
\Bigg(\E\bigg[\Big|\frac{f(W_1)+\cdots+f(W_n)}{n}-\E[f(W_1)]\Big|^q\bigg]\Bigg)^{\frac{1}{q}}\lesssim \sqrt{\frac{q}{(1-\lambda_\mathbf{W})n}}\cdot \Big(\E
\big[|f(W_1)|^q\big]\Big)^{\frac{1}{q}}.
\end{equation}
$\phantom{1}$
\end{theorem}
The (standard) asymptotic notation $\lesssim$ that appears in~\eqref{eq:q in rhs1} (as well as throughout the ensuing discussion) means the following. Given two quantities $\alpha,\beta\in [0,\infty)$, the notation $\alpha\lesssim \beta$ stands  for the assertion that there exists a universal constant $C\in (0,\infty)$ for which $\alpha\le C\beta$; this is also denoted by $\beta\gtrsim \alpha$.

The conclusion~\eqref{eq:q in rhs1} of Theorem~\ref{momentbound} with the random variables $f(W_1),\ldots,f(W_n)$ replaced by i.i.d.\ random variables coincides with the classical Marcinkiewicz--Zygmund inequality~\cite{MZ37}. Our contribution here is therefore to generalize this statement to random variables that are (images of) stationary Markov chains with a spectral gap; the i.i.d.\ setting is the special case $A=E_\pi$ of Theorem~\ref{momentbound}. The bound~\eqref{eq:q in rhs1} is optimal; see Remark~\ref{rem:optimal} below. A variant of Theorem~\ref{momentbound} when $1\le q\le 2$ appears in Remark~\ref{rem:q<2} below.

The precursor (and inspiration) of Theorem~\ref{momentbound} is the following theorem of Gillman~\cite{Gil93,G98}.

\begin{theorem}\label{thm:quote gillman}
Suppose that $\mathbf{W}=\{W_t\}_{t=1}^{\infty}$ is a  stationary Markov chain whose state space is $[N]$ and  with $\lambda_\mathbf{W}<1$.  Then, every $f:[N]\to \R$ satisfies the following inequality for every $n\in \N$ and every $q\ge 2$.
\begin{equation}\label{eq:infty in rhs}
\Bigg(\E\bigg[\Big|\frac{f(W_1)+\cdots+f(W_n)}{n}-\E[f(W_1)]\Big|^q\bigg]\Bigg)^{\frac{1}{q}}\lesssim \sqrt{\frac{q}{(1-\lambda_\mathbf{W})n}}\cdot \max\left\{|f(1)|,\ldots,|f(N)|\right\}.
\end{equation}
\end{theorem}
Note that Theorem~\ref{thm:quote gillman} is typically stated in the literature as the following concentration inequality, which is commonly called the {\em expander Chernoff bound}.
\begin{equation}\label{eq:tail version}
\forall\, a>0,\qquad \Pr \bigg[\Big\vert\frac{f(W_1)+\cdots+f(W_n)}{n}-\E[f(W_1)]\Big|\ge a\max_{j\in [N]} |f(j)|\bigg] \lesssim e^{-c(1-\lambda_\mathbf{W})n a^2},
\end{equation}
where $c>0$ is a universal constant. The equivalence of~\eqref{eq:infty in rhs} and~\eqref{eq:tail version} is standard;  $\eqref{eq:infty in rhs}\implies\eqref{eq:tail version}$ is checked by applying Markov's inequality and optimizing over $q$, and $\eqref{eq:tail version}\implies\eqref{eq:infty in rhs}$ follows by straightforward integration (both implications appear in~Proposition~2.5.2 of the textbook~\cite{Ver18}). The same use of Markov's inequality shows mutatis mutandis that Theorem~\ref{momentbound} implies the following concentration phenomenon.
\begin{corollary}\label{cor:t in range}There is a universal constant $c>0$ with the following property. Suppose that $\mathbf{W}=\{W_t\}_{t=1}^{\infty}$ is a  stationary Markov chain whose state space is $[N]$ and  with $\lambda_\mathbf{W}<1$. Then, every $f:[N]\to \R$ satisfies the following inequality for every $n\in \N$, every $q\ge 2$ and every $0<a\le \sqrt{q/((1-\lambda_{\mathbf{W}})n)}$.
\begin{equation*}
 \Pr \bigg[\Big|\frac{f(W_1)+\cdots+f(W_n)}{n}-\E[f(W_1)]\Big|\ge a\Big(\E
\big[|f(W_1)|^q\big]\Big)^{\frac{1}{q}}\bigg] \lesssim e^{-c(1-\lambda_\mathbf{W})n a^2}.
\end{equation*}
\end{corollary}

\begin{remark}{\em Kloeckner investigated in~\cite{Kloeckner19} the question of obtaining concentration bounds such as~\eqref{eq:tail version} with the $L_\infty$ norm  $\max_{j\in [N]} |f(j)|$ replaced by other norms of $f$.  As discussed in~\cite[Remark~2.2]{Kloeckner19}, the results of~\cite{Kloeckner19} hold in a setting that imposes structural hypotheses on the aforementioned norm of the ``observable'' $f$ which notably excludes its $L_q(\pi_\mathbf{W})$ norm (which appears  in the right-hand side of the bound~\eqref{eq:q in rhs1} that we prove here), but it is noted in~\cite[Remark~2.2]{Kloeckner19} that ``classically one only makes moment assumptions on the observable.'' Corollary~\ref{cor:t in range} addresses this question, though note that~\cite{Kloeckner19} also covers settings that are not treated here.}\end{remark}

The new bound~\eqref{eq:q in rhs1}  that we obtain differs from Gillman's estimate~\eqref{eq:infty in rhs} only in the replacement of the worst-case bound on $f$ in the right-hand side of~\eqref{eq:infty in rhs} by an average-case bound. Rather than being merely a quantitative enhancement, this improvement  has conceptual significance which we achieve through a reasoning  that  differs substantially from the proof of~\eqref{eq:tail version} in~\cite{Gil93,G98}, as well as the several other proofs of~\eqref{eq:tail version} and its variants that appeared in the literature~\cite{D95, K97, L98, LCP04,K07,Wagner08,ChungLLM12,P15,GLSS17,FJS18,Kloeckner19} (our approach  was recently used in~\cite{RR17,R18}).

Assuming a bound on the $q$'th moment of $f$  is the appropriate setting for bounding the $q$'th moment of $f(W_1)+\cdots+f(W_n)$. This compatibility of the left-hand side of~\eqref{eq:q in rhs1}  and the right-hand side of~\eqref{eq:q in rhs1}  allows the resulting inequality to tensorize so as to yield dimension-independent vector-valued statements. Specifically, for any measure space $(\Omega,\mu)$, if $f:[N]\to L_q(\mu)$, then by applying~\eqref{eq:q in rhs1} to the real-valued mapping $(i\in [N])\mapsto f(i)(\omega)$ for each $\omega\in \Omega$, and then integrating the ($q$'th power of) the resulting point-wise inequality, we see that (under the assumptions of Theorem~\ref{momentbound}),
\begin{equation}\label{eq:q in rhs Lq power q}
\Bigg(\E\bigg[\Big\|\frac{f(W_1)+\cdots+f(W_n)}{n}-\E[f(W_1)]\Big\|_{L_q(\mu)}^q\bigg]\Bigg)^{\frac{1}{q}} \lesssim \sqrt{\frac{q}{(1-\lambda_\mathbf{W})n}}\cdot \Big(\E
\big[\|f(W_1)\|_{L_q(\mu)}^q\big]\Big)^{\frac{1}{q}}.
\end{equation}
The following Hilbertian statement is a consequence of~\eqref{eq:q in rhs Lq power q} that deserves to be stated separately.

\begin{corollary}\label{cor:hilbert moment}
Suppose that $\mathbf{W}=\{W_t\}_{t=1}^{\infty}$ is a  stationary Markov chain whose state space is $[N]$ and  with $\lambda_\mathbf{W}<1$. Let $(H,\|\cdot\|_H)$ be a Hilbert space. The following bound holds for all $n\in \N$,  $q\ge 2$ and $f:[N]\to H$.
\begin{equation}\label{eq:q in rhs}
\Bigg(\E\bigg[\Big\|\frac{f(W_1)+\cdots+f(W_n)}{n}-\E[f(W_1)]\Big\|_H^q\bigg]\Bigg)^{\frac{1}{q}} \lesssim \sqrt{\frac{q}{(1-\lambda_\mathbf{W})n}}\cdot \Big(\E
\big[\|f(W_1)\|_H^q\big]\Big)^{\frac{1}{q}}.
\end{equation}
\end{corollary}
Corollary~\ref{cor:hilbert moment} is nothing more than~\eqref{eq:q in rhs Lq power q} applied to an isometric copy of $H$ in $L_q(\mu)$, which is known to exist by~\cite[Chapter~12]{Ban32} (see also the exposition in, e.g., the textbook~\cite[Proposition~6.4.12]{AK16}).

Since $\E
[\|f(W_1)\|_H^q]\le \max_{j\in [N]} \|f(j)\|_H^q$, the following corollary is a consequence Corollary~\ref{cor:hilbert moment} through the usual application of Markov's inequality and then an optimization over $q$.
\begin{corollary}[Hilbert space-valued expander Chernoff bound]\label{cor:Hilbert-dim independent}
There is a universal constant $c>0$ with the following property. Suppose that $\mathbf{W}=\{W_t\}_{t=1}^{\infty}$ is a  stationary Markov chain whose state space is $[N]$ and  with $\lambda_\mathbf{W}<1$. Let $(H,\|\cdot\|_H)$ be a Hilbert space. If $f:[N]\to H$, then  for all $n\in \N$ and $a>0$ we have
\begin{equation}\label{eq:tail version hilbert1}
\Pr \bigg[\Big\|\frac{1}{n}\sum_{i=1}^nf(W_i)-\E[f(W_1)]\Big\|_H\ge a\max_{j\in [N]} \|f(j)\|_H\bigg] \lesssim e^{-c(1-\lambda_\mathbf{W})n a^2}.
\end{equation}
\end{corollary}

\begin{remark}{\em Kargin studied~\cite{K07} the vector-valued setting of Gillman's theorem for functions that take values in the $m$-dimensional  Euclidean space $\ell_2^m$. The statement that is obtained in~\cite{K07}  is the same as that of Corollary~\ref{cor:Hilbert-dim independent}, except that it is dimension-dependent; specifically, with the implicit constant in~\eqref{eq:tail version hilbert1} growing to $\infty$ exponentially with $m$. Thus, the main new feature of Corollary~\ref{cor:Hilbert-dim independent} is that it is dimension-independent. Obtaining such a bound was a main question that~\cite{K07}  left open; see~\cite[Section~4]{K07}.
}\end{remark}

Observe that estimates such as~\eqref{eq:q in rhs Lq power q} can be interpreted as  bounds on the operator norm of  a certain linear operator between vector-valued $L_q$-spaces. Specifically, suppose that $(X,\|\cdot\|_X)$ is a Banach space. Let $\mathbf{W}=\{W_t\}_{t=1}^{\infty}$ be a  stationary Markov chain whose state space is $[N]$ and with $\lambda_\mathbf{W}<1$. Denote (as before) the stationary measure of $\mathbf{W}$ by $\pi_\mathbf{W}$ and let the transition matrix of $\mathbf{W}$ be $A=(a_{ij})\in \M_N(\R)$. For each $n\in \N$ denote  the associated probability measure on the trajectories of length $n$ by $\tau^n_\mathbf{W}:[N]^n\to [0,1]$. Thus, $\tau^n_\mathbf{W}$ is the probability measure on $[N]^n$ that is given by $\tau^1_\mathbf{W}=\pi_\mathbf{W}$ if $n=1$, and for $n\ge 2$,
$$
\forall(i_1,\ldots,i_n)\in [N]^n,\qquad \tau^n_\mathbf{W}(i_1,\ldots,i_n)\eqdef \Pr\big[(W_1,\ldots,W_n)=(i_1,\ldots,i_n)\big]=\pi_\mathbf{W}(i_1)a_{i_1i_2}a_{i_2i_3}\cdots a_{i_{n-1}i_n}.
$$

Define a linear operator $T_X:L_q(\pi_\mathbf{W};X)\to L_q(\tau^n_\mathbf{W};X)$ by setting for $f:[N]\to X$,
\begin{equation}\label{eq:def TX}
\forall(i_1,\ldots,i_n)\in [N]^n,\qquad T_Xf(i_1,\ldots,i_n)\eqdef \frac{1}{n}\sum_{k=1}^n f(i_k)-\sum_{j=1}^N\pi_\mathbf{W}(j)f(j) \in X.
\end{equation}
Here, and in what follows, we are using standard notation for vector-valued Lebesgue--Bochner spaces, though throughout we will need to consider only finitely supported measures, in which case measurability issues do not need to be discussed. So, if $(S,\sigma)$ is a probability space with $|S|<\infty$, then the Banach space $L_q(\sigma;X)$ is the vector space of all mapping $\psi:S\to X$, equipped with the norm
$$
\|\psi\|_{L_q(\sigma;X)}=\bigg(\sum_{s\in S} \sigma(s)\|\psi(s)\|_X^q\bigg)^{\frac{1}{q}}.
$$

The validity of~\eqref{eq:q in rhs Lq power q}  under the assumptions of Theorem~\ref{momentbound} is the same as the operator norm bound
\begin{equation}\label{eq:q to interpolate}
\|T_{L_q(\mu)}\|_{L_q\left(\pi_\mathbf{W};L_q(\mu)\right)\to L_q\left(\tau_{\mathbf{W}}^n;L_q(\mu)\right)}\lesssim \sqrt{\frac{q}{(1-\lambda_\mathbf{W})n}}.
\end{equation}
In the same vein, Corollary~\ref{cor:hilbert moment} is (under the same assumptions) the same as
\begin{equation}\label{eq:2 to interpolate}
\|T_{L_2(\mu)}\|_{L_q\left(\pi_\mathbf{W};L_2(\mu)\right)\to L_q\left(\tau_{\mathbf{W}}^n;L_2(\mu)\right)}\lesssim \sqrt{\frac{q}{(1-\lambda_\mathbf{W})n}}.
\end{equation}
By Calder\'on's vector-valued extension~\cite{Cal64}  of the Riesz--Thorin~\cite{Rie27,Tho48} interpolation theorem (see the monograph~\cite{BL76} for  background on complex interpolation; the specific statement that we are using here is a combination of Theorem~4.1.2 and Theorem~5.1.2 in~\cite{BL76}), it follows from~\eqref{eq:q to interpolate} and~\eqref{eq:2 to interpolate} that for every $p\in [2,q]$ we have
\begin{align*}
\|T_{L_p(\mu)}\|_{L_q\left(\pi_\mathbf{W};L_p(\mu)\right)\to L_q\left(\tau_{\mathbf{W}}^n;L_p(\mu)\right)}
\lesssim \sqrt{\frac{q}{(1-\lambda_\mathbf{W})n}}.
\end{align*}
%\begin{align*}
%\|T_{L_p(\mu)}\|_{L_q\left(\pi_\mathbf{W};L_p(\mu)\right)\to L_q\left(\tau_{\mathbf{W}}^n;L_p(\mu)\right)}&\le %\|T_{L_2(\mu)}\|_{L_2\left(\pi_\mathbf{W};L_2(\mu)\right)\to %L_2\left(\tau_{\mathbf{W}}^n;L_2(\mu)\right)}^{\frac{2(q-p)}{p(q-2)}}\|T_{L_q(\mu)}\|_{L_q\left(\pi_\mathbf{W};L_q(\mu)\right)\to %L_q\left(\tau_{\mathbf{W}}^n;L_q(\mu)\right)}^{\frac{q(p-2)}{p(q-2)}}\\
%&\lesssim \sqrt{\frac{qn}{1-\lambda_\mathbf{W}}}.
%\end{align*}
We record this conclusion as the following generalization of Corollary~\ref{cor:hilbert moment} and Corollary~\ref{cor:Hilbert-dim independent}.

\begin{corollary}\label{cor:q>p} Suppose that $p\ge 2$ and that $(\Omega,\mu)$ is a measure space. Let $\mathbf{W}=\{W_t\}_{t=1}^{\infty}$ be a  stationary Markov chain whose state space is $[N]$ and  with $\lambda_\mathbf{W}<1$. If $f:[N]\to L_p(\mu)$, then for all $n\in \N$ and $q\ge p$,
\begin{equation}\label{eq:q in rhs Lp power q}
\Bigg(\E\bigg[\Big\|\frac{f(W_1)+\cdots+f(W_n)}{n}-\E[f(W_1)]\Big\|_{L_p(\mu)}^q\bigg]\Bigg)^{\frac{1}{q}} \lesssim \sqrt{\frac{q}{(1-\lambda_\mathbf{W})n}}\cdot \Big(\E
\big[\|f(W_1)\|_{L_p(\mu)}^q\big]\Big)^{\frac{1}{q}}.
\end{equation}
Consequently, by the usual combination of~\eqref{eq:q in rhs Lp power q} with Markov's inequality, followed by optimization over $q\ge p$, there exists a universal constant $c\in (0,\infty)$ such that
\begin{equation}\label{eq:tail version hilbert}
\forall\, a>0,\qquad \Pr \bigg[\Big\|\frac{f(W_1)+\cdots+f(W_n)}{n}-\E[f(W_1)]\Big\|_{L_p(\mu)}\ge a\max_{j\in [N]} \|f(j)\|_{L_p(\mu)}\bigg] \lesssim e^{p-c(1-\lambda_\mathbf{W})n a^2}.
\end{equation}
\end{corollary}

\begin{remark}\label{rem:q<2} {\em By convexity we have $\|T_\R\|_{L_1(\pi_\mathbf{W})\to L_1(\tau_\mathbf{W}^n)}\le 2$, since it is evident from~\eqref{eq:def TX} that the operator in question is the difference of two averaging operators. By interpolating this (trivial) estimate with the case $q=2$ of Theorem~\ref{momentbound} using the (scalar-valued) Riesz--Thorin interpolation theorem as above, we arrive at the following variant of Theorem~\ref{momentbound} in the range $1\le q\le 2$, which holds under the same assumptions.
\begin{equation}\label{eq:q<2}
\Bigg(\E\bigg[\Big|\frac{f(W_1)+\cdots+f(W_n)}{n}-\E[f(W_1)]\Big|^q\bigg]\Bigg)^{\frac{1}{q}}\lesssim \bigg(\frac{1}{(1-\lambda_\mathbf{W})n}\bigg)^{1-\frac{1}{q}}\cdot \Big(\E
\big[|f(W_1)|^q\big]\Big)^{\frac{1}{q}}.
\end{equation}

Observe that when the Markov chain $\mathbf{W}$ is reversible,  the case $q=2$ of~\eqref{eq:q in rhs1} is a quadratic inequality that could be directly verified in a straightforward manner by expanding both sides  in an orhtonormal eigenbasis of the transition matrix of $\mathbf{W}$. The more substantial content of Theorem~\ref{momentbound}  is therefore the case $q>2$, which does not lend itself to such linear-algebraic reasoning.}
\end{remark}

\begin{remark}\label{rem:optimal}{\em Both~\eqref{eq:q in rhs1} and~\eqref{eq:q<2} are sharp (up to the implicit universal constant factors) for large enough $n\in \N$. This is seen by examining the following family of Markov chains. For every $\e,\lambda\in (0,1)$ consider  the two-state Markov chain $\mathbf{W}(\lambda,\e)$ whose transition matrix equals
\begin{equation}\label{eq:2 by 2}
\begin{pmatrix} 1-(1-\lambda)(1-\e)& (1-\lambda)(1-\e)\\
              (1-\lambda)\e & 1-(1-\lambda)\e
                       \end{pmatrix}=\lambda {I}_2+(1-\lambda)E_{\pi(\e)}\in  \M_2(\R), \end{equation}
where $I_2$ is the $2$-by-$2$ identity matrix and $\pi(\e)=(\e,1-\e)\in \triangle^{\!1}$. Then $\pi_{\mathbf{W}(\lambda,\e)}=\pi(\e)$ and $\lambda_{\mathbf{W}(\lambda,\e)}=\lambda$. 

The optimality of~\eqref{eq:q in rhs1} is exhibited by taking $\e=\frac12$ and $f:\{1,2\}\to \R$ that is given by $f(1)=1=-f(2)$. In this case, it is elementary to check that if $n\ge q/(1-\lambda)$, then both sides of~\eqref{eq:q in rhs1} are within universal constant multiples of each other. Next, the optimality of~\eqref{eq:q<2} is exhibited by considering $f:\{1,2\}\to \R$ that is given by $f(1)=1$ and $f(2)=0$.  In this case, it is elementary to check that if $n\ge 1/(1-\lambda)$, then for small enough $\e>0$ both sides of~\eqref{eq:q<2} are within universal constant multiples of each other. The routine computations that verify these assertions are omitted.}
\end{remark}

\begin{remark} {\em The above discussion raises the question of understanding what is required from a Banach space $(X,\|\cdot\|_X)$ so that the ``Gillman phenomenon'' for stationary Markov chains (or variants thereof)   would hold for $X$-valued mappings. The present work obtains the first examples (notably, Hilbert space) of such theorems in infinite dimensions (equivalently, dimension-independent bounds). However, much more remains to be understood here. This matter is pursued in the forthcoming work~\cite{Nao19}, where it is explained how it relates to central themes in Banach space theory. Further infinite dimensional statements are derived in~\cite{Nao19}, including a treatment of~\eqref{eq:q in rhs Lp power q} in the range $2\le q<p$ which is not covered in Corollary~\ref{cor:q>p},  through an approach that is entirely different from our reasoning here.
}
\end{remark}

We end the Introduction by noting that the above results have an equivalent dual formulation that is worthwhile to work out explicitly. Given a Banach space $(X,\|\cdot\|_X)$, the operator $T_X$ that is given in~\eqref{eq:def TX} has norm $K>0$ from $L_q(\pi_\mathbf{W};X)$ to $L_q(\tau^n_\mathbf{W};X)$ if and only if its adjoint $T^*_X$ has norm $K$ from $L_{q^*}(\tau^n_\mathbf{W};X^*)$ to $L_{q^*}(\pi_\mathbf{W};X^*)$, where $q^*=q/(q-1)$. This leads to the following dual formulation of Corollary~\ref{cor:q>p}, whose derivation is  a mechanical unravelling of the definitions (the straightforward details are omitted).

\begin{corollary}[adjoint of~\eqref{eq:q in rhs Lp power q}] Let $\mathbf{W}=\{W_t\}_{t=1}^{\infty}$ be a  stationary Markov chain whose state space is $[N]$ and  with $\lambda_\mathbf{W}<1$. Fix $n\in \N$ and $p,q\in (1,2]$ with $q\le p$. For every measure space $(\Omega,\mu)$ and   $F:[N]^n\to L_p(\mu)$,
\begin{align*}
\Bigg(\E\bigg[\Big\|\frac{1}{n}\sum_{i=1}^n \E \Big[F(W_1,\ldots,W_n)\Big| W_i\Big]&-\E [F(W_1,\ldots,W_n)]\Big\|_{L_p(\mu)}^q\bigg]\Bigg)^{\frac{1}{q}}\\&\lesssim \frac{1}{\sqrt{(q-1)(1-\lambda_\mathbf{W})n}} \cdot\Big(\E\big[\|F(W_1,\ldots,W_n)\|_{L_p(\mu)}^q\big]\Big)^{\frac{1}{q}}.
\end{align*}
\end{corollary}

%\begin{comment} for $q\in [1,2]$
%\begin{equation}\label{eq:q in rhs<2}
%\Big(\E\big[|f(W_1)+\cdots+f(W_n)|^q\big]\Big)^{\frac{1}{q}} \lesssim \frac{n^{\frac{1}{q}}}{(1-\lambda_\mathbf{W})^{1-\frac{1}{q}}}\cdot \Big(\E
%\big[|f(W_1)|^q\big]\Big)^{\frac{1}{q}}.
%\end{equation}
%\end{comment}

\section{Proof of Theorem~\ref{momentbound}}\label{sec:proof}

Suppose from now on that we are in the setting of Theorem~\ref{momentbound}. We will write for simplicity $\lambda=\lambda_\mathbf{W}<1$ and $\pi=\pi_\mathbf{W}\in \Nsimplex$. We will also let $A=(a_{ij})\in \M_N(\R)$ be the transition matrix of $\mathbf{W}$.

It suffices to prove~\eqref{eq:q in rhs1} when $f:[N]\to \R$ satisfies $\E[f(W_1)]=0$. Indeed, this could be then applied to the centered function $f-\E[f(W_1)]$ to yield the estimate
\begin{align}\label{eq:center}
\begin{split}
\Bigg(\E\bigg[\Big|\frac{f(W_1)+\cdots+f(W_n)}{n}-\E[f(W_1)]\Big|^q\bigg]\Bigg)^{\frac{1}{q}}&\lesssim \sqrt{\frac{q}{(1-\lambda_\mathbf{W})n}}\cdot \Big(\E
\big[|f(W_1)-\E[f(W_1)]|^q\big]\Big)^{\frac{1}{q}}\\&\le 2\sqrt{\frac{q}{(1-\lambda_\mathbf{W})n}}\cdot \Big(\E
\big[|f(W_1)|^q\big]\Big)^{\frac{1}{q}},
\end{split}
\end{align}
where the last step is the triangle inequality in $L_q(\pi)$. So, assume from now on that $\E[f(W_1)]=0$. It will be convenient to define $u\in \R^N$ by setting $u_i=f(i)$ for all $i\in [N]$. The assumption on $f$ becomes $\sum_{i=1}^N\pi_iu_i=0$. Below,  we will denote   the diagonal matrix whose diagonal is $u$ by $U\in \M_N(\R)$, i.e.,
$$U\eqdef
\begin{pmatrix} u_1 & 0 & \dots&0 \\
  0 & u_2& \ddots & \vdots\\
            \vdots & \ddots & \ddots &0\\
              0 & \dots  &0&u_N
                       \end{pmatrix}\eqdef\begin{pmatrix} f(1) & 0 & \dots&0 \\
  0 & f(2)& \ddots & \vdots\\
            \vdots & \ddots & \ddots &0\\
              0 & \dots  &0&f(N)
                       \end{pmatrix}.$$

\begin{lemma}\label{lem:increasing} For every $m\in \N$ we have
$$
\E \Big[\big(f(W_1)+\cdots+f(W_n)\big)^{2m}\Big]\le (2m)!\sum_{\substack{v_0,\ldots,v_{2m-1} \in \N\cup \{0\}\\ v_0+\cdots+v_{2m-1} \le n-1}}
\big\|U A^{v_1}UA^{v_2}\cdots UA^{v_{2m-1}}u \big\|_{L_1(\pi)}.
$$
\end{lemma}

\begin{proof} Let $V_{2m}$ be the set of all those vectors in $w \in [n]^{2m}$ that satisfy $1 \le w_1 \leq w_2 \leq \cdots \leq w_{2m} \le n$. Observe that by the Markov property and stationarity, for  every $w\in V_{2m}$ we have the following identity.
\begin{align*}\label{eq:commutative identity}
\begin{split}
\E\left[\prod_{i=1}^{2m} f(W_{w_i})\right]&=\sum_{j\in [N]^{2m}} \pi_{j_1}A^{w_2-w_1}_{j_1j_2}A^{w_3-w_2}_{j_2j_3}\cdots A^{w_{2m}-w_{2m-1}}_{i_{2m-1}i_{2m}}\prod_{k=1}^{2m} u_{j_k}\\&=\sum_{j\in [N]^{2m}} \pi_{j_1} (UA^{w_2-w_1})_{j_1j_2}(UA^{w_3-w_2})_{j_2j_3}\cdots (UA^{w_{2m}-w_{2m-1}})_{j_{2m-1}j_{2m}}u_{j_{2m}}\\
&=\sum_{i\in [N]} \pi_i (UA^{w_2-w_1}UA^{w_3-w_2}\cdots UA^{w_{2m}-w_{2m-1}}u)_i.
\end{split}
\end{align*}
So, by expanding the $(2m)$'th power of $f(W_1)+\cdots+f(W_n)$ and arranging the indices in increasing order,
\begin{align*}
\E \Big[\big(f(W_1)+\cdots+f(W_n)\big)^{2m}\Big]
&\le(2m)!\sum_{w\in V_{2m}}\Bigg| \E\left[\prod_{i=1}^{2m} f(W_{w_i})\right] \Bigg|\\
&\le (2m)! \sum_{w\in V_{2m}}\|UA^{w_2-w_1}UA^{w_3-w_2}\cdots UA^{w_{2m}-w_{2m-1}}u\|_{L_1(\pi)}.\qedhere
\end{align*}
\end{proof}

\begin{remark} {\em It is worthwhile to note in passing that while the proof of Lemma~\ref{lem:increasing}  relies on what may seem to be innocuous identities, the crucial step that rearranged the factors so that their indices are  increasing is inherently commutative, and this is what obstructs the direct use of the ensuing proof for matrix-valued functions, namely the setting of~\cite{WX08,GLSS17}; alternative routes are taken in~\cite{GLSS17,Nao19} but it would be interesting to investigate if a more careful reasoning along the lines of the present work could be used to treat the setting of functions that take values in Schatten--von Neuman trace classes.}
\end{remark}

Towards  bounding from above each of the terms $\| U A^{v_1}UA^{v_2}\cdots UA^{v_{2m-1}}u \|_{L_1(\pi)}$ from Lemma~\ref{lem:increasing}, we record the following iterative application of H\"older's inequality and the definition of operator norms.

\begin{lemma}\label{lem:alternate} Fix $k\in \N$ and $q\ge k+1$. Then, for every  $T_1,\ldots,T_{k}\in \M_N(\R)$ we have
\begin{equation*}\label{eq:2m-1}
\left\|UT_1UT_2\cdots UT_{k}u\right\|_{L_1(\pi)}
\leq
\|u\|_{L_{q}(\pi)}^{k+1} \prod_{j=1}^{k} \|T_j\|_{L_{\frac{2q}{q+k+1-2j}}(\pi) \rightarrow L_{\frac{2q}{q+k+1-2j}}(\pi)}.
\end{equation*}
\end{lemma}

\begin{proof} Suppose that $\alpha(1),\ldots,\alpha(k+1)\ge 1$ satisfy $\frac{1}{\alpha(1)}+\cdots+\frac{1}{\alpha(k+1)} \le 1$.
We claim that
\begin{equation}\label{eq:alpha beta}
\left\|UT_1UT_2\cdots UT_{k}u\right\|_{L_{\beta(0)}(\pi)}\le \bigg(\prod_{i=1}^{k+1}\|u\|_{L_{\alpha(i)}(\pi)}\bigg)\prod_{j=1}^k\|T_j\|_{L_{\beta(j)}(\pi)\to L_{\beta(j)}(\pi)},
\end{equation}
where $\beta(0),\ldots,\beta(k)\ge 1$ are defined by  $\frac{1}{\beta(j)}=\frac{1}{\alpha(j+1)}+\cdots+\frac{1}{\alpha(k+1)}$.
The proof of~\eqref{eq:alpha beta} is by induction on $k$.

The case $k=0$ is tautological.   For the induction step, since $\frac{1}{\beta(0)}=\frac{1}{\alpha(1)}+\frac{1}{\beta(1)}$, by H\"older's inequality,
\begin{equation}\label{eq:holder first}
\left\|UT_1UT_2\cdots UT_{k}u\right\|_{L_{\beta(0)}(\pi)}\le \|u\|_{L_{\alpha(1)}(\pi)}\left\|T_1UT_2\cdots UT_{k}u\right\|_{L_{\beta(1)}(\pi)}.
\end{equation}
By the definition of the operator norm $\|T_1\|_{L_{\beta(1)}(\pi)\to L_{\beta(1)}(\pi)}$ we have,
\begin{equation}\label{eq:op}
\left\|T_1UT_2\cdots UT_{k}u\right\|_{L_{\beta(1)}(\pi)}\le \|T_1\|_{L_{\beta(1)}(\pi)\to L_{\beta(1)}(\pi)}\left\|UT_2\cdots UT_{k}u\right\|_{L_{\beta(1)}(\pi)}.
\end{equation}
Now~\eqref{eq:alpha beta} follows by   combining~\eqref{eq:holder first} and~\eqref{eq:op} with the inductive hypothesis.

Choose $\alpha(1)=\alpha(k+1)=\frac{2q}{q-k+1}$ and $\alpha(2)=\cdots=\alpha(k)=q$. So,
$$
\forall\, j\in [k],\qquad \beta(j)=\frac{1}{\frac{1}{\alpha(j+1)}+\cdots+\frac{1}{\alpha(k+1)}}=\frac{1}{\frac{k-j}{q}+\frac{q-k+1}{2q}}=\frac{2q}{q+k+1-2j},
$$
and $\beta(0)=1$.
Hence,
with this specific setting of the parameters the bound~\eqref{eq:alpha beta} becomes
$$
\left\|UT_1UT_2\cdots UT_{k}u\right\|_{L_{1}(\pi)}
\le
\|u\|_{L_{\frac{2q}{q-k+1}}(\pi)}^2\|u\|_{L_{q}(\pi)}^{k-1} \prod_{j=1}^k \|T_j\|_{L_{\frac{2q}{q+k+1-2j}}(\pi) \rightarrow L_{\frac{2q}{q+k+1-2j}}(\pi)}.
$$
It remains to note that since $q\ge k+1$ we have $\frac{2q}{q-k+1}\le q$, and therefore $\|u\|_{L_{\frac{2q}{q-k+1}}(\pi)}\le \|u\|_{L_q(\pi)}$.
\end{proof}

Fix $m\in \N$. Throughout what follows, it will be notationally convenient to consider each Boolean vector $s\in \{0, 1\}^{2m-1}$ as an infinite vector in $\{0, 1\}^\Z$  whose entries vanish on $\Z\setminus [2m-1]$, namely we use the convention $s_i=s_j=0$ for $i\le 0$ and $j\ge 2m$. Let $S_{2m-1}\subset \{0, 1\}^{2m-1}$ be all  those Boolean vectors of length $2m-1$ with no two consecutive $0$s, and with $s_{2m-1} = 1$, i.e.,
$$
S_{2m-1}\eqdef \bigcap_{j=1}^{2m-1} \Big\{s\in \{0, 1\}^{2m-1}:\  (s_j,s_{j+1})\neq (0,0)\Big\}.
$$
For each  $j\in [2m-1]$ and  $s \in S_{2m-1}$ that satisfy $s_j=1$, we define a quantity
$p(s,j) \ge 1$ in the following way. Consider the consecutive run of $1$s in $s$ to which
$j$ belongs, and let $i_1(s,j)$ and $i_2(s,j)$ be the first and last indices of this run, respectively. Formally,
\begin{equation}\label{eq:def i1i2}
i_1(s,j) \eqdef \max\big\{i\in \{...,j-2,j-1\}:\ s_i=0\big\}+1\quad\mathrm{and}\quad i_2(s,j)\eqdef \min\big\{i\in \{j+1,j+2,\ldots\}:\ s_i=0\big\}-1.
\end{equation}
With this notation, write
\begin{equation}\label{eq:def psj}
p(s,j)\eqdef \frac{4m}{2m+i_1(s,j)+i_2(s,j)-2j}.
\end{equation}

\begin{lem}\label{lem:holderapplication} For every
$T_1, \ldots, T_{2m-1} \in \M_N(\R)$,
\begin{align}\label{eq:splittingj}
\big\| U(T_1+E_\pi)U(T_2+E_\pi)\cdots U(T_{2m-1}+E_\pi)u\big\|_{{L_1(\pi)}}
\leq
\|u\|_{L_{2m}(\pi)}^{2m} \sum_{s \in S_{2m-1}}\prod_{\substack{j\in [2m-1]\\s_j=1}}\big\|T_j\big\|_{L_{p(s,j)}(\pi) \rightarrow L_{p(s,j)}(\pi) }.
\end{align}
\end{lem}

\begin{proof}
For each $j\in [2m-1]$, write $T_{j, 0} = E_{\pi}$ and $T_{j, 1} = T_j$. Observe that
\begin{equation}\label{eq:off Sk}
\forall\, s\in \{0,1\}^{2m-1}\setminus S_{2m-1},\qquad UT_{1,s_1}UT_{2,s_2}\cdots UT_{2m-1,s_{2m-1}}u=0.
\end{equation}
Indeed, if $s\in \{0,1\}^{2m-1}\setminus S_{2m-1}$, then  either $s_{2m-1}=0$, in which case $T_{2m-1,s_{2m-1}}u=E_\pi u=\mathbf{0}\in \R^N$, or $s_j=s_{j+1}=0$ for some $j\in [2m-2]$, in which case $T_{j,s_j}UT_{j+1,s_{j+1}}=E_\pi UE_\pi=\mathbf{0}\in \M_N(\R)$, where both identities  are equivalent to the assumption $\sum_{i=1}^N\pi_iu_i=0$. Now,
\begin{align}
\begin{split}
\big\| U(T_1+E_\pi)U(T_2+E_\pi)\cdots U(T_{2m-1}+E_\pi)u\big\|_{{L_1(\pi)}}
&=
\Big\|\sum_{s \in \{0, 1\}^{2m-1}} UT_{1,s_1}UT_{2,s_2}\cdots UT_{2m-1,s_{2m-1}}u\Big\|_{L_1(\pi)} \\
&\le \sum_{s \in \{0, 1\}^{2m-1}} \big\|UT_{1,s_1}UT_{2,s_2}\cdots UT_{2m-1,s_{2m-1}}u\big\|_{L_1(\pi)} \\
&\!\!\!\stackrel{\eqref{eq:off Sk}}{=}
\sum_{s \in S_{2m-1}} \big\|UT_{1,s_1}UT_{2,s_2}\cdots UT_{2m-1,s_{2m-1}}u \big\|_{L_1(\pi)}.
\end{split} \label{eq:mainclaiminmono}
\end{align}

Fix $s \in S_{2m-1}$ and let $1 \le r_1 < r_2 < \cdots < r_\ell < 2m-1$ be all of the indices at which $s$ vanishes. Define $R_1,\ldots,R_{\ell+1}\in \M_N(\R)$ by setting
$$R_1\eqdef (UT_{1})(U T_{2})\cdots (UT_{r_1-1}),\qquad R_{\ell+1}\eqdef (UT_{r_{\ell}+1})(U T_{r_{\ell}+2})\cdots (UT_{2m-1}),$$ and
$$R_\kappa\eqdef (UT_{r_{\kappa-1}+1})(U T_{r_{\kappa-1}+2})\cdots (UT_{r_{\kappa}-1})$$ for $\kappa\in \{2,\ldots,\ell\}$. Using the fact that $UE_\pi v=\left(\sum_{i=1}^N \pi_iv_i\right)u$ for every $v\in \R^N$,  we have the following identity.
$$
UT_{1,s_1}UT_{2,s_2}\cdots UT_{2m-1,s_{2m-1}}u=R_1(UE_\pi) R_2(UE_\pi) R_3\cdots (UE_\pi) R_{\ell+1}u =\Bigg(\prod_{\kappa=2}^{\ell+1}\sum_{i=1}^N\pi_i(R_{\kappa}u)_i \Bigg)R_1u.
$$
Consequently,
\begin{equation}\label{eq:Rj product}
\big\|UT_{1,s_1}UT_{2,s_2}\cdots UT_{2m-1,s_{2m-1}}u \big\|_{L_1(\pi)}= \|R_1u\|_{L_1(\pi)}\prod_{\kappa=2}^{\ell+1}\Big|\sum_{i=1}^N\pi_i(R_{\kappa}u)_i\Big|\le \prod_{\kappa=1}^{\ell+1} \|R_\kappa u\|_{L_1(\pi)}.
\end{equation}

Next, by Lemma~\ref{lem:alternate} with $q=2m$ and $k=r_1-1$ we have
\begin{align*}
\|R_1u\|_{L_1(\pi)}&= \|(UT_{1})(U T_{2})\cdots (UT_{r_1-1})u\|_{L_1(\pi)}\\&\le \|u\|_{L_{2m}(\pi)}^{r_1}\prod_{j=1}^{r_1-1} \|T_j\|_{L_{\frac{4m}{2m+r_1-2j}}(\pi) \rightarrow L_{\frac{4m}{2m+r_1-2j}}(\pi)}\\&\!\!\!\stackrel{\eqref{eq:def psj}}{=} \|u\|_{L_{2m}(\pi)}^{r_1}\prod_{j=1}^{r_1-1} \|T_j\|_{L_{p(s,j)}(\pi) \rightarrow L_{p(s,j)}(\pi)}.
\end{align*}
In the same vein, for every $k\in \{2,\ldots,\ell\}$,
$$
\|R_\kappa u\|_{L_1(\pi)}\le \|u\|_{L_{2m}(\pi)}^{r_\kappa-r_{\kappa-1}}\prod_{j=r_{\kappa-1}+1}^{r_\kappa-1} \|T_j\|_{L_{p(s,j)}(\pi) \rightarrow L_{p(s,j)}(\pi)},
$$
and also
$$
\|R_{\ell+1} u\|_{L_1(\pi)}\le \|u\|_{L_{2m}(\pi)}^{2m-r_\ell}\prod_{j=r_\ell+1}^{2m-1} \|T_j\|_{L_{p(s,j)}(\pi) \rightarrow L_{p(s,j)}(\pi)}.
$$
We therefore have
\begin{equation}\label{eq:tekescope exponents}
\prod_{\kappa=1}^{\ell+1} \|R_\kappa u\|_{L_1(\pi)}\le \|u\|_{L_{2m}(\pi)}^{2m}\prod_{\substack{j\in [2m-1]\\s_j=1}}\big\|T_j\big\|_{L_{p(s,j)}(\pi) \rightarrow L_{p(s,j)}(\pi) }.
\end{equation}
By substituting~\eqref{eq:tekescope exponents} into~\eqref{eq:Rj product} and then substituting the resulting estimate into~\eqref{eq:mainclaiminmono}, we arrive at~\eqref{eq:splittingj}.
\end{proof}

In light of Lemma~\ref{lem:increasing}, the following lemma is highly relevant to our goal of proving Theorem~\ref{momentbound}.

\begin{lem}\label{finbnew}
Suppose that $m\in \N$ satisfies $em\le n(1-\lambda)$. Then,
\begin{equation}\label{eq:secondlemmamoments}
\bigg(\sum_{\substack{v_0,\ldots,v_{2m-1} \in \N\cup \{0\}\\ v_0+\cdots+v_{2m-1} \le n-1}}
\big\|U A^{v_1}UA^{v_2}\cdots UA^{v_{2m-1}}u \big\|_{L_1(\pi)}\bigg)^{\frac{1}{2m}}\lesssim \frac{\sqrt{n/m}}{\sqrt{1-\lambda}}\|u\|_{L_{2m}(\pi)}.
\end{equation}
\end{lem}
\begin{proof} Fix $v_0,\ldots,v_{2m-1} \in \N\cup \{0\}$ and denote  $T_j=A^{v_j}-E_\pi$ for every $j\in \{0,\ldots,2m-1\}$. Then,
\begin{align}
\begin{split}
\big\|U A^{v_1}UA^{v_2}\cdots UA^{v_{2m-1}}u \big\|_{L_1(\pi)}&=\big\|U (T_1+E_\pi)U(T_2+E_\pi)\cdots U(T_{2m-1}+E_\pi)u \big\|_{L_1(\pi)}\\
&\le \|u\|_{L_{2m}(\pi)}^{2m} \sum_{s \in S_{2m-1}}\prod_{\substack{j\in [2m-1]\\s_j=1}}\big\|T_j\big\|_{L_{p(s,j)}(\pi) \rightarrow L_{p(s,j)}(\pi) },
\end{split}\label{eq:use holder lemma}
\end{align}
where the last step of~\eqref{eq:use holder lemma} is an application of Lemma~\ref{lem:holderapplication}.

Fixing $j\in \{0,\ldots,2m-1\}$, note that $AE_\pi=E_\pi$ since $A$ is stochastic and the columns of $E_\pi$ are constant, and also $E_\pi A=E_\pi$ since $A$ is $\pi$-stationary. Consequently $T_j=A^{v_j}-E_\pi= (A-E_\pi)^{v_j}$. So, for every $p\ge 1$,
\begin{equation}\label{eq:take vj power out}
\|T_j\|_{L_p(\pi)\to L_p(\pi)}=\big\|(A-E_\pi)^{v_j}\big\|_{L_p(\pi)\to L_p(\pi)}\le \|A-E_\pi\|_{L_p(\pi)\to L_p(\pi)}^{v_j}.
\end{equation}
By definition, $\|A-E_\pi\|_{L_2(\pi)\to L_2(\pi)}=\lambda$. As $A$ and $E_\pi$ are averaging operators, by convexity and the triangle inequality  $\|A-E_\pi\|_{L_r(\pi)\to L_r(\pi)}\le \|A\|_{L_r(\pi)\to L_r(\pi)}+\|E_\pi\|_{L_r(\pi)\to L_r(\pi)}= 2$ for all $r\ge 1$. By the Riesz--Thorin interpolation theorem~\cite{Rie27,Tho48} (see e.g.~Chapter~IV in the textbook~\cite{Kat04}), this implies that
\begin{equation}\label{eq:interpolated gap}
\|A-E_\pi\|_{L_p(\pi)\to L_p(\pi)}\le 2\lambda^{2\min\left\{\frac{1}{p},1-\frac{1}{p}\right\}}.
\end{equation}

A substitution of~\eqref{eq:interpolated gap} into~\eqref{eq:take vj power out}, followed by a substitution of the resulting bound into~\eqref{eq:use holder lemma} shows that in order to prove the desired inequality~\eqref{eq:secondlemmamoments} it suffices to establish the following estimate.
\begin{equation}\label{eq:first reduction big sum}
\bigg(\sum_{s \in S_{2m-1}}\sum_{\substack{v_0,\ldots,v_{2m-1} \in \N\cup \{0\}\\ v_0+\cdots+v_{2m-1} \le n-1}}\prod_{\substack{j\in [2m-1]\\s_j=1}}\lambda^{\beta(s,j)v_j}\bigg)^{\frac{1}{m}}\lesssim \frac{n}{m(1-\lambda)},
\end{equation}
where for every $s\in S_{2m-1}$ and $j\in [2m-1]$ such that $s_j=1$, we denote
\begin{equation}\label{eq:beta formula}
\beta(s,j)\eqdef 2\min\left\{\frac{1}{p(s,j)},1-\frac{1}{p(s,j)}\right\}\stackrel{\eqref{eq:def psj}}{=} 1-\frac{|i_1(s,j)+i_2(s,j)-2j|}{2m}.
\end{equation}
%Recall that $i_1(s,j)$ and $i_2(s,j)$ are, respectively, the first and last indices of the consecutive run of $1$s in $s$ to which
%$j$ belongs.

Fix some $s \in S_{2m-1}$. Denote $Q_0=\{j\in [2m-1]: s_j=0\}$ and $Q_1=[2m-1]\setminus Q_0$. Thus $|Q_0|+|Q_1|=2m-1$ and by the definition of $S_{2m-1}$ we have $|Q_1| \geq m$. With this notation, we have the following bound.
\begin{align*}
\sum_{\substack{v_0,\ldots,v_{2m-1} \in \N\cup \{0\}\\ v_0+\cdots+v_{2m-1} \le n-1}}&\prod_{\substack{j\in [2m-1]\\s_j=1}}\lambda^{\beta(s,j)v_j}
\\&= \sum_{\substack{(v_i)_{i\in \{0\}\cup Q_0}\in (\N\cup\{0\})^{\{0\}\cup Q_0}\\ \sum_{i\in \{0\}\cup Q_0} v_i\le n-1}}\sum_{\substack{(v_j)_{j\in Q_1}\in (\N\cup\{0\})^{Q_1}\\ \sum_{j\in Q_1}v_j\le n-1-\sum_{i\in \{0\}\cup Q_0} v_i}}\prod_{j\in Q_1}\lambda^{\beta(s,j)v_j}
\\&\le
\Big|\Big\{(v_i)_{i\in \{0\}\cup Q_0}\in (\N\cup\{0\})^{\{0\}\cup Q_0}:\sum_{i\in \{0\}\cup Q_0} v_i\le n-1\Big\}\Big|\cdot \sum_{(v_j)_{j\in Q_1}\in (\N\cup\{0\})^{Q_1}}\prod_{j\in Q_1}\lambda^{\beta(s,j)v_j} \\
&= \sum_{\ell=0}^{n-1}\binom{|Q_0|+\ell}{|Q_0|}\prod_{j\in Q_1}\sum_{i=0}^\infty \lambda^{\beta(s,j)i}\\&= \binom{|Q_0|+n}{|Q_0|+1}\prod_{j\in Q_1}\frac{1}{1-\lambda^{\beta(s,j)}}.
\end{align*}
By the elementary inequality $1-\lambda^\beta\ge \beta(1-\lambda)$, which holds for every $\lambda,\beta\in [0,1]$, it follows from this that
\begin{align}\label{eq:bernoulli}
\begin{split}
\sum_{\substack{v_0,\ldots,v_{2m-1} \in \N\cup \{0\}\\ v_0+\cdots+v_{2m-1} \le n-1}}\prod_{\substack{j\in [2m-1]\\s_j=1}}\lambda^{\beta(s,j)v_j}
&\le
\frac{1}{(1-\lambda)^{|Q_1|}}\binom{|Q_0|+n}{|Q_0|+1}\prod_{j\in Q_1}\frac{1}{\beta(s,j)} \\
&=
\frac{(1-\lambda)^{|Q_0|+1}}{(1-\lambda)^{2m}}\binom{|Q_0|+n}{|Q_0|+1}\prod_{j\in Q_1} \frac{1}{\beta(s,j)} \\
&\lesssim
\frac{e^{O(m)}}{(1-\lambda)^{2m}}\left(\frac{(1-\lambda)n}{|Q_0|+1}\right)^{|Q_0|+1}\prod_{j\in Q_1} \frac{1}{\beta(s,j)},
\end{split}
\end{align}
where the last step follows from a straightforward application of Stirling's formula. Consider the function $\psi:[0,\infty)\to [0,\infty)$ that is given by  $\psi(z)=((1-\lambda)n/z)^z$. Then $(\log\psi(z))'=\log ((1-\lambda) n/(ez))$. Hence, $\psi$ is increasing on the interval $[0,(1-\lambda)n/e]$. But $|Q_0|+1=2m-|Q_1|\le m\le (1-\lambda)n/e$, by the assumption on $m$  in the statement of Lemma~\ref{finbnew}. Hence $\psi(|Q_0|+1)\le \psi(m)$, and therefore
\begin{align}\label{eq:psi m}
%\begin{split}
\frac{1}{(1-\lambda)^{2m}}\left(\frac{(1-\lambda)n}{|Q_0|+1}\right)^{|Q_0|+1}
=
\frac{\psi(|Q_0|+1)}{(1-\lambda)^{2m}}
\le
\frac{\psi(m)}{(1-\lambda)^{2m}}
=
\frac{(n/m)^m}{(1-\lambda)^{m}}.
%\end{split}
\end{align}

We will show next that
\begin{equation}\label{eq:Q1 prod goal}
\prod_{j\in Q_1} \frac{1}{\beta(s,j)}
\stackrel{\eqref{eq:beta formula}}{=}
\prod_{j\in Q_1} \left(1-\frac{|i_1(s,j)+i_2(s,j)-2j|}{2m}\right)^{-1}\le e^{O(m)}.
\end{equation}
In combination with \eqref{eq:bernoulli} and~\eqref{eq:psi m}, this would imply the desired inequality~\eqref{eq:first reduction big sum} because $|S_{2m-1}|\le e^{O(m)}$.

For each $j \in Q_1$ with $i_2(s,j) - i_1(s,j) \le \frac{3m}{2}$ (i.e.,
the consecutive run of $1$s in $s$ to which $j$ belongs is of length at most $1+\frac{3m}{2}$), we have
$|i_1(s,j)+i_2(s,j)-2j| \le \frac{3m}{2}$ and therefore its contribution to the product in~\eqref{eq:Q1 prod goal} is at most $4$.
So, \eqref{eq:Q1 prod goal} holds if there are no runs of $1$s in $s$ of length greater than $\frac{3m}{2}$.
Otherwise, there is exactly one run of $1$s in $s$  of length $d > \frac{3m}{2}$, and its contribution to the product in~\eqref{eq:Q1 prod goal} equals
\begin{align*}
\prod_{i=0}^{\left\lfloor\frac{d-1}{2}\right\rfloor} \left(\frac{2m}{2m-d+1+2i}\right)^2
\le
\prod_{i=0}^{2\left\lfloor\frac{d-1}{2}\right\rfloor} \frac{2m}{2m-d+1+i}
\le
\prod_{k=2}^{2m} \frac{2m}{k} =\frac{(2m)^{2m-1}}{(2m)!}
\le
e^{O(m)},
\end{align*}
where the last step follows from Stirling's formula. This proves our goal~\eqref{eq:Q1 prod goal}.
\end{proof}

\begin{proof}[Completion of the proof of Theorem~\ref{momentbound}] By the triangle inequality in $L_q$ (and stationarity) we have
$$
\Big(\E\big[|f(W_1)+\cdots+f(W_n)|^q\big]\Big)^{\frac{1}{q}} \le \Big(\E
\big[|f(W_1)|^q\big]\Big)^{\frac{1}{q}}+\cdots+\Big(\E
\big[|f(W_n)|^q\big]\Big)^{\frac{1}{q}}=n\Big(\E
\big[|f(W_1)|^q\big]\Big)^{\frac{1}{q}}.
$$
This bound implies the desired estimate~\eqref{eq:q in rhs1} when $q\gtrsim (1-\lambda)n$, so we may assume from now on that $q\le (1-\lambda)n/e$. Let $m\in \N$ be the largest integer such that $2m\le q$. Then, $m,m+1\le q\le  (1-\lambda)n/e$, so the conclusion of Lemma~\ref{finbnew} holds for both $m$ and $m+1$. By Lemma~\ref{lem:increasing} (and Stirling's formula), this gives

\begin{align*}
\bigg(\E \Big[\big(f(W_1)+\cdots+f(W_n)\big)^{2m}\Big]\bigg)^{\frac{1}{2m}}
\lesssim
\sqrt{\frac{nm}{1-\lambda}}\Big(\E \big[|f(W_1)|^{2m}\big]\Big)^{\frac{1}{2m}}
\le
\sqrt{\frac{nq}{1-\lambda}}\Big(\E \big[|f(W_1)|^{2m}\big]\Big)^{\frac{1}{2m}},
\end{align*}
and similarly
\begin{align*}
\bigg(\E \Big[\big(f(W_1)+\cdots+f(W_n)\big)^{2(m+1)}\Big]\bigg)^{\frac{1}{2(m+1)}}
 \lesssim
\sqrt{\frac{nq}{1-\lambda}}\Big(\E \big[|f(W_1)|^{2(m+1)}\big]\Big)^{\frac{1}{2(m+1)}}.
\end{align*}
As in~\eqref{eq:center}, it follows from these bounds (which we derived under the assumption $\E[f(W_1)]=0$) that the norm of the operator $T_\R$ that is given in~\eqref{eq:def TX} is bounded by a universal constant multiple of $\sqrt{q/((1-\lambda)n)}$ both from $L_{2m}(\pi)$ to $L_{2m}(\pi)$ and from $L_{2(m+1)}(\pi)$ to $L_{2(m+1)}(\pi)$. Since $2m\le q\le 2(m+1)$, another application of the Riesz--Thorin theorem gives that the norm of $T_\R$ from $L_{q}(\pi)$ to $L_{q}(\pi)$ is also bounded by a universal constant multiple of $\sqrt{q/((1-\lambda)n)}$. This is precisely the desired  bound~\eqref{eq:q in rhs1}.
\end{proof}

\bibliographystyle{alphaabbrvprelim}
\bibliography{GillmanRelaxation}

\end{document}